\documentclass{amsart}

\usepackage{amssymb,enumitem,graphicx,mathtools,subfig}
\usepackage[a4paper]{geometry}
\newtheorem{theorem}{Theorem}[section]
\newtheorem{lemma}[theorem]{Lemma}

\theoremstyle{definition}

\newtheorem{definition}[theorem]{Definition}
\newtheorem{example}[theorem]{Example}

\numberwithin{equation}{section}
\numberwithin{figure}{section}

\def \isnatural {\in\mathbb{N}}
\def\R{\mathbb{R}}
\def\N{\mathbb{N}}
\def\Z{\mathbb{Z}}
\def\Q{\mathbb{Q}}

\def\good{F_{\sigma\delta}}
\def\No{\mathbb{N}_0}

\renewcommand{\geq}{\geqslant}
\renewcommand{\leq}{\leqslant}

\newcommand{\tef}{transcendental entire function}
\newcommand\qfor{\text{for }}

\DeclarePairedDelimiter{\floor}{\langle}{\rangle}

\makeatletter
\def\blfootnote{\xdef\@thefnmark{}\@footnotetext}

\makeatother

\begin{document}

\title[Escaping sets of continuous functions]{Escaping sets of continuous functions}
\author{Ian Short, \, David J. Sixsmith}
\address{Department of Mathematics and Statistics \\ The Open University \\   Milton Keynes MK7 6AA\\   UK}
\email{Ian.Short@open.ac.uk}
\address{School of Mathematical Sciences \\ University of Nottingham \\ Nottingham
NG7 2RD \\ UK}
\email{David.Sixsmith@nottingham.ac.uk}

\thanks{The second author was supported by Engineering and Physical Sciences Research Council grant EP/L019841/1.}
%
%
%
%
\begin{abstract}
Our objective is to determine which subsets of $\R^d$ arise as escaping sets of continuous functions from $\R^d$ to itself. We obtain partial answers to this problem, particularly in one dimension, and in the case of open sets. We give a number of examples to show that the situation in one dimension is quite different from the situation in higher dimensions. Our results demonstrate that this problem is both interesting and perhaps surprisingly complicated.
\end{abstract}
\maketitle
%
%
%
%

\section{Introduction}
\subsection{Background}
Given a function $f : \R^d \to \R^d$, we denote the $k$th iterate of $f$ by $f^k$, for any positive integer $k$, and we define the \emph{escaping set} $I(f)$ of $f$ to be
\[
I(f) = \{ x \in \R^d : f^k(x) \rightarrow\infty \text{ as } k\rightarrow\infty \}.
\]
The escaping set of an entire function is a familiar object in complex dynamics. Such sets were first studied in generality by Eremenko \cite{MR1102727}, who showed that if $f$ is a transcendental entire function, then $I(f)$ is never empty and  the closure of $I(f)$ has no bounded components. Eremenko conjectured that, in fact, $I(f)$ has no bounded components; this conjecture remains open and has stimulated much further investigation.

In \cite{MR2448586}, escaping sets of quasiregular functions were studied; we refer to, for example, \cite{MR950174} for a definition of a quasiregular function. It was shown in \cite{MR2448586} that if $f : \R^d \to \R^d$ is a quasiregular function, then, providing $f$ satisfies a certain growth condition, the escaping set $I(f)$ necessarily contains an unbounded component. Also, an example was given of a quasiregular function with an essential singularity at infinity for which the closure of the escaping set has a bounded component. This shows that Eremenko's conjecture fails for quasiregular functions.

The escaping set can be defined for \emph{any} function from $\R^d$ to itself -- the function need not be complex differentiable or quasiregular -- so it is reasonable to study escaping sets in a more general context. In this paper, we investigate the structure of escaping sets of \emph{continuous} functions from $\R^d$ to itself (our functions need not even be surjective). More explicitly, we address the following question: \emph{given a subset $S$ of $\R^d$, does there exist a continuous function $f : \R^d \to \R^d$ such that $S = I(f)$?} We give examples to show that this question, in its full generality, is difficult. We make headway with some special cases, particularly in one dimension, and highlight some interesting differences between the theory of escaping sets of continuous functions and the theory of escaping sets of transcendental entire functions. We are not aware of any other study of escaping sets in this generality. 

To facilitate our discussion, we make the following definition.

\begin{definition}
Suppose that $S$ is a subset of $\R^d$ . Then we say that $S$ is an \emph{$I_d$-set} if there is a continuous function $f : \R^d \to \R^d$ such that $S = I(f)$.
\end{definition}

Notice that the collection of $I_d$-sets is invariant under homeomorphisms, in the following sense. Let us denote the set $\R^d\cup\{\infty\}$ by $\overline{\R^d}$, and endow it with the usual topology of a one-point compactification. Suppose that $S$ and $S'$ are subsets of $\R^d$, and there is a homeomorphism $\phi : \overline{\R^d} \to \overline{\R^d}$ such that $\phi(S) = S'$ and $\phi(\infty) = \infty$. Then $S$ is an $I_d$-set if and only if $S'$ is an $I_d$-set. This follows from the fact that if $S = I(f)$, then $S' = I(\phi \circ f \circ \phi^{-1})$.

If $f$ is continuous, then $I(f)$ is either empty or unbounded. Moreover, we can write
\[
I(f) = \bigcap_{p=1}^\infty \bigcup_{q=1}^\infty \bigcap_{m \geq q} \{ x : |f^m(x)| \geq p \}.
\]
 It follows that a necessary condition for a non-empty set $S$ to be an $I_d$-set is that $S$ is an unbounded $\good$ set. However, there are many examples in this paper of unbounded $\good$ sets that are not $I_d$-sets. In order to obtain sufficient conditions for sets to be $I_d$-sets, we must impose further restrictions on them; first we look at open sets, and then sets in one dimension. The reader may care to peruse Figure~\ref{fig:sets} and speculate as to which of the six open planar sets, shown in black, are $I_2$-sets; all will be revealed later.

\begin{figure}[ht]
\subfloat[]{\includegraphics[width=.28\textwidth]{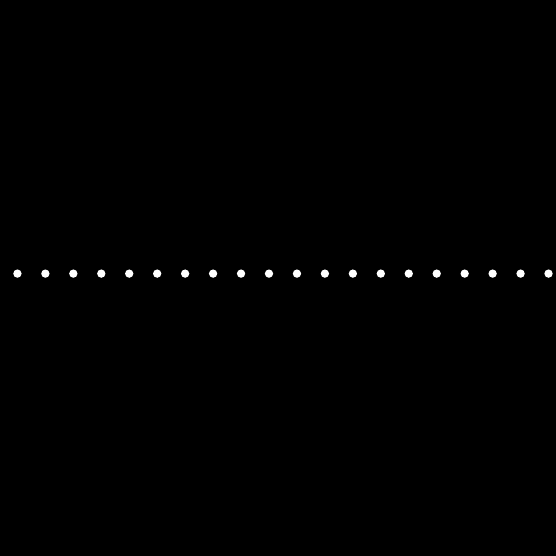}}\quad
\subfloat[]{\includegraphics[width=.28\textwidth]{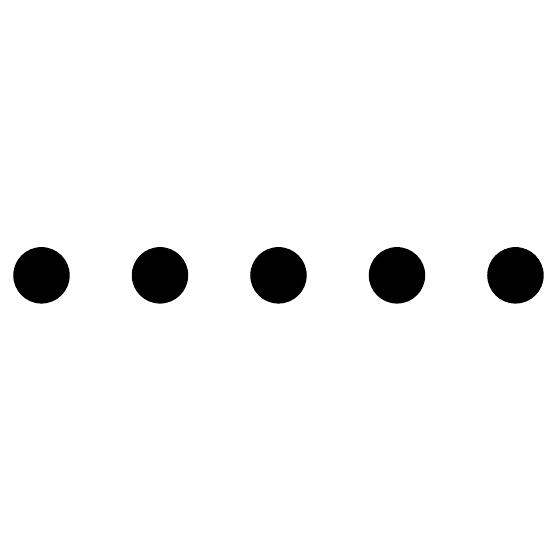}}\quad
\subfloat[]{\includegraphics[width=.28\textwidth]{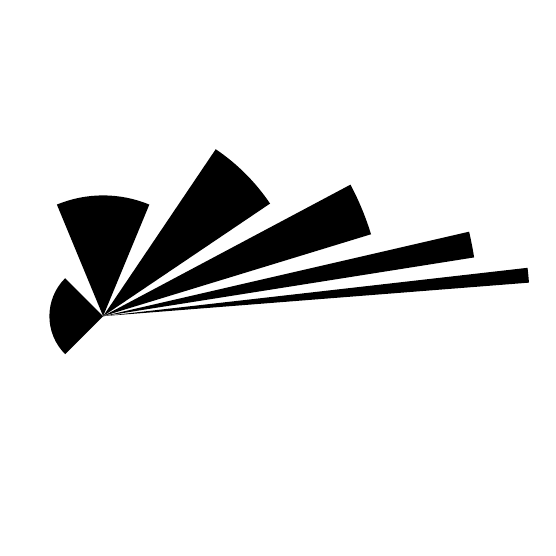}}\hfill
\subfloat[]{\includegraphics[width=.28\textwidth]{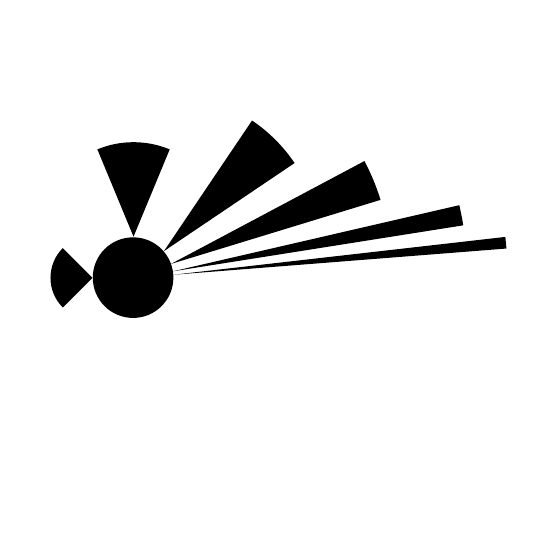}}\quad
\subfloat[]{\includegraphics[width=.28\textwidth]{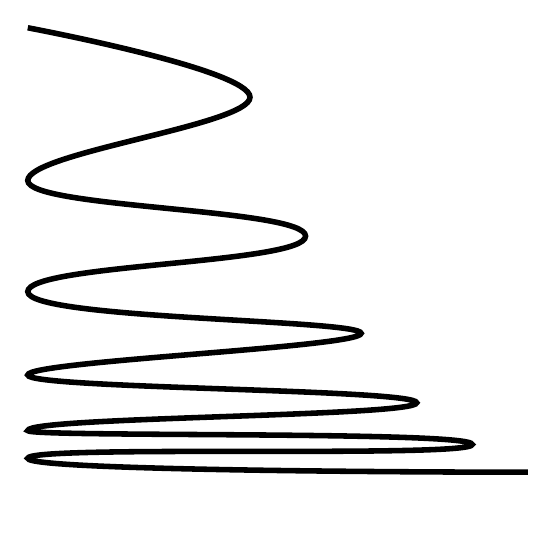}}\quad
\subfloat[]{\includegraphics[width=.28\textwidth]{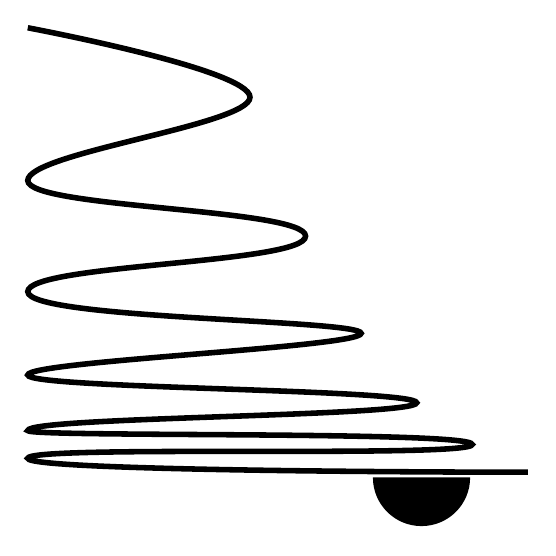}}\hfill
\caption{(a) $\mathbb{R}^2\setminus (\mathbb{Z}\times\{0\})$; (b) $\bigcup_{n\in\mathbb{Z}} D(n,\tfrac14)$; (c) $\bigcup_{n=1}^\infty S(0,n,\pi/2^{n-1})$;\\ (d) $D(0,1)\cup \bigcup_{n=1}^\infty S(1,n,\pi/2^{n-1})$; (e) an infinite snake-like region that accumulates on $[0,+\infty)$; (f) the same snake-like region together with the set $D(5,1)\cap \{(x,y)\,:\,y<0\}$}
\label{fig:sets}
\end{figure}

In the caption of Figure~\ref{fig:sets}, we denote the open disc with radius $r$ and centre $c$ by $D(c,r)$. Also, for positive numbers $r$, $s$ and $\phi$, we write
\[
S(r,s,\phi) = \left\{(r\cos\phi+t\cos\theta,r\sin\phi+t\sin\theta)\,:\, |\theta-\phi|<\phi/4, \ 0<t<s\right\},
\]
which is an open sector of the disc $D((r\cos\phi,r\sin\phi),s)$ of angle $\phi/2$. We remark that none of the open sets displayed in the figure are the escaping sets of transcendental entire functions, as the escaping sets of such functions are never open. This follows from the observation that if, on the contrary, $f$ is a transcendental entire function and $I(f)$ is open, then the \emph{Julia set} of $f$, which is equal to $\partial I(f)$, does not meet $I(f)$, which contradicts a result of \cite{MR1102727}; we refer to \cite{MR1216719} for further background and definitions.

It is useful to observe now that, if $f:\R^d\to\R^d$, then $I(f)$ is \emph{completely invariant}, in the sense that $f(I(f))\subset I(f)$ and $f^{-1}(I(f)) \subset I(f)$.

\subsection{Open sets}

The results in this subsection are about classifying \emph{open} $I_d$-sets. We follow the convention here (and later on) that if $d$ is not specified, then it can take any positive integer value. We define an \emph{unbounded curve} to be a continuous injective function $\gamma:[0,+\infty)\to \R^d$ such that the image set $\gamma([0,+\infty))$ is unbounded. We say that infinity is \emph{accessible} from an unbounded open subset $U$ of $\R^d$ if there is an unbounded curve $\gamma:[0,+\infty)\to U$ that does not  accumulate at any point of $\R^d$, in the sense that $\gamma(t)\to \infty$ as $t\to\infty$. Such a curve is a homeomorphism onto its image.

\begin{theorem}\label{theo:hasanunboundedcomponent}
Any open, unbounded subset of $\R^d$ from which infinity is accessible is an $I_d$-set.
\end{theorem}

We remark that Baker \cite[Theorem 2]{MR980793} showed that infinity is always accessible in a so-called \emph{Baker domain} of a {\tef}; a Baker domain is an open set in which the iterates tend to infinity locally uniformly, and so always lies in the escaping set. Later on we will see two examples, Examples~\ref{example:snake} and \ref{example:failsnake}, of unbounded subsets of $\R^2$ from which infinity is \emph{not} accessible. In the first example, the unbounded set \emph{is} an $I_2$-set, and in the second, the unbounded set is \emph{not} an $I_2$-set. The set in the second example, however, has two components; it is not a domain. The question of whether every unbounded domain in $\R^2$ is an $I_2$-set remains open.

Next we consider open sets without unbounded components. There certainly are open sets without unbounded components that are $I_2$-sets: Example~\ref{example.open} (to come) provides an example of one. The following theorem gives necessary conditions for an open set without unbounded components to be an $I_d$-set.

\begin{theorem}\label{theo:hasnounboundedcomponent}
Suppose that $S$ is a non-empty open subset of $\R^d$ that has no unbounded components. If $S$ is an $I_d$-set, then there exist a sequence $(V_k)_{k\isnatural}$ of components of $S$ and $r>0$, such that
\begin{enumerate}[label=\emph{(\roman*)}]
\item $\sup\{|x|\,:\, x \in V_k\} \rightarrow +\infty$ as $k\rightarrow\infty$;\label{t21}
\item $\inf\{|x|\,:\,x \in V_k\} \leq r$, for $k\isnatural$.\label{t22}
\end{enumerate}
\end{theorem}

We show later (see Example~\ref{example.nastyone}) that the conditions of Theorem~\ref{theo:hasnounboundedcomponent}  are not sufficient  for an open set with no unbounded components to be an $I_d$-set.

\subsection{One dimension} \label{d1case}

Let us now consider escaping sets in one dimension. We can immediately classify the open $I_1$-sets using Theorems~\ref{theo:hasanunboundedcomponent} and \ref{theo:hasnounboundedcomponent}, because it is impossible for an open set in $\R$ to satisfy the conditions of Theorem~\ref{theo:hasnounboundedcomponent}.

\begin{theorem} \label{theo:nounb}
A non-empty open subset of $\R$ is an $I_1$-set if and only if it has an unbounded component.
\end{theorem}

Our next result shows that even in one dimension, escaping sets can have a complicated structure.

\begin{theorem}\label{theo:R1irrational}
The set of irrationals is an $I_1$-set.
\end{theorem}

It is well known (see, for example, \cite{Sierpinski1920}) that if $X$ is a countable, dense subset of $\R$, then there is a homeomorphism $\phi:\overline{\R}\to\overline{\R}$ with $\phi(\infty)=\infty$ and $\phi(X)=\Q$. Therefore the complement of $X$ in $\R$ is an $I_1$-set.

In view of Theorem~\ref{theo:R1irrational}, it is natural to ask whether the set of integers, or the set of rational numbers, is an $I_1$-set. It follows from the next result that this is not the case.

\begin{theorem}\label{theo:R1uncountable}
Suppose that $S$ is an $I_1$-set, and that $x \in S$. Then every neighbourhood of $x$ contains uncountably many elements of $S$.
\end{theorem}

Our final result in the one-dimensional setting can be seen as step towards classifying $I_1$-sets.

\begin{theorem}\label{theo:R1open}
If $S$ is an $I_1$-set, then all components of the complement of $S$ are closed.
\end{theorem}

\subsection{Higher dimensions}

Classifying $I_d$-sets when $d\geq 2$ is far more difficult than classifying $I_1$-sets. We give several examples that illustrate the differences between the one- and higher-dimensional cases. Here are brief descriptions of three such examples.

\begin{enumerate}[label=(\roman*),leftmargin=20pt,topsep=0pt,itemsep=0pt]
\item In contrast to Theorem~\ref{theo:R1open}, we give an $I_2$-set, the complement of which is a domain (Example~\ref{example.halfplane}).
\item In contrast to Theorem~\ref{theo:R1uncountable}, we give a countable  nowhere dense set, which is an $I_2$-set (Example~\ref{example.Z}).
\item In contrast to the previous example, we give a countable nowhere dense set, the complement of which is an $I_2$-set (Example~\ref{example.Zcomp}).
\end{enumerate}

Although, for simplicity, all these examples are in two dimensions, it is clear that they can be extended to any dimension greater than one.

\subsection{Structure of the paper}

In Section~\ref{sec:Rmore} we prove Theorems~\ref{theo:hasanunboundedcomponent} and~\ref{theo:hasnounboundedcomponent}. Then, in Section~\ref{sec:R1}, we prove Theorems~\ref{theo:R1irrational}, \ref{theo:R1uncountable} and~\ref{theo:R1open}. Finally, in Section~\ref{sec:examples} we give all the examples mentioned above.

%
%
%
%
%
%
\section{Open sets}\label{sec:Rmore}

We require the following lemma, which is a variation, in Euclidean space, of the Tietze extension theorem; see, for example, \cite[Theorem 15.8]{MR0264581}. In the proof of the lemma (and elsewhere in the paper) we use the notation $\text{dist}(x, K)$ to represent the Euclidean distance from a point $x$ to a set $K$ in $\R^d$; that is, $\text{dist}(x, K)= \inf_{y\in K} |x-y|$.

\begin{lemma}\label{lemma:open}
Suppose that $K$ is a closed subset of $\R^d$ and $\phi : K \to [0,+\infty)$ is a continuous function. Then there exists a continuous function $\psi : \R^d \to [0,+\infty)$ such that $\psi|_K = \phi$ and $\phi^{-1}(0) = \psi^{-1}(0)$.
\end{lemma}
\begin{proof}
Let
\begin{equation*}
\psi(x) =
\begin{cases}
\phi(x), &\qfor x \in K, \\
\inf_{y \in K} \left( \phi(y) + \dfrac{|x-y|}{\operatorname{dist}(x, K)} - 1 \right) + \operatorname{dist}(x, K), &\qfor x \notin K.
\end{cases}
\end{equation*}
It is easy to check that $\psi$ has the required properties.
\end{proof}
\begin{proof}[Proof of Theorem~\ref{theo:hasanunboundedcomponent}]
Suppose that $U$ is an unbounded component of an open set $S$ in $\R^d$, and that $\gamma:[0,+\infty)\to U$ is an unbounded curve that does not accumulate at any point of $\R^d$. Clearly $\R^d$ is an $I_d$-set, so we can assume that $U \ne \R^d$. Choose a point $x_1$ in the image of $\gamma$. Consider the set of straight lines emanating from $x_1$. At least one of these lines must meet $\partial U$. Let the first point where one such line meets $\partial U$ be the point $x_0$. Without loss of generality -- replacing $x_1$ with another point of $\gamma$ if necessary -- we can assume that $\gamma$ does not intersect the open line segment from $x_0$ to $x_1$. After reparameterizing we can assume that $\gamma(1)=x_1$.

Now let $\Gamma : [0,+\infty) \to U\cup\{x_0\}$ be an unbounded curve comprised of the line segment from $x_0$ to $x_1$ followed by the part of $\gamma$ from $x_1$ to $\infty$; explicitly,
\[
\Gamma(t) =
\begin{cases}
tx_1+(1-t)x_0, &\qfor 0\leq t\leq 1,\\
\gamma(t), &\qfor t\geq 1.
\end{cases}
\]
We denote the image of $\Gamma$ by $|\Gamma|$; that is, $|\Gamma| = \Gamma([0,+\infty))$.

We now define a continuous function $f : \R^d \to \R^d$ such that $I(f) = S$. First set $f(x) = x_0$, for $x \notin S$. Thus we certainly have that $I(f) \subset S$.

Next we define $f$ on $|\Gamma|$. Define $h : [0,+\infty)\to[0,+\infty)$ by $h(t) = 2t$. Then set
\[
f(x) = (\Gamma \circ h \circ \Gamma^{-1})(x),\quad \text{for}\quad x \in |\Gamma|.
\]
This definition of $f$ agrees with the previous definition at the point $x_0$. As $f$ is conjugate to $h$ on $|\Gamma|$, it follows from the fact that $\Gamma(t)\rightarrow\infty$ as $t\rightarrow\infty$ that $|\Gamma|\setminus\{x_0\} \subset I(f)$.

Next, if $V$ is a component of $S$ other than $U$, then we define
\[
f(x) = \Gamma\left(\operatorname{dist}(x,\partial V)\right),\quad \text{for}\quad  x \in V.
\]
Since $f(x)\rightarrow\Gamma(0) = x_0$ as $x\rightarrow v$, for $v\in\partial V$, this definition ensures that $f$ is continuous in $\R^d\setminus U$. Moreover, $f$ maps all points of $S\setminus U$ to $|\Gamma|\setminus\{x_0\}$. Hence $S\setminus U \subset I(f)$.

It remains to define $f$ in $U \setminus |\Gamma|$. To do this we first use Lemma~\ref{lemma:open} to construct an auxiliary function. Let $K = (\R^d \setminus U) \cup |\Gamma|$. Note that, since $\Gamma$ does not accumulate in $\R^d$, $K$ is closed. Define a continuous function $\phi : K \to [0,+\infty)$ by
\[
\phi(x) =
\begin{cases}
\Gamma^{-1}(x), &\qfor x \in |\Gamma|, \\
0, &\text{otherwise}.
\end{cases}
\]
It follows, by Lemma~\ref{lemma:open}, that there exists a continuous function $\psi~:~\R^d~\to~[0,+\infty)$ such that $\psi|_K = \phi$ and $\psi^{-1}(0) = \phi^{-1}(0) = \R^d \setminus U$. We can now see that
\begin{enumerate}[label=(\alph*)]
\item $\psi(x) = 0$, for $x \notin U$;
\item $\psi(x) > 0$, for $x \in U$;
\item $\psi(x) = \Gamma^{-1}(x)$, for $x \in |\Gamma|$.
\end{enumerate}
To finish, we define $f(x) = (\Gamma \circ h \circ \psi)(x)$, for $x \in U \setminus |\Gamma|$. Conditions (a) and (c)  ensure that this definition of $f$ is consistent with previous definitions, and that $f : \R^d \to \R^d$ is continuous. Condition (b)  ensures that if $x \in U$, then $f(x) \in |\Gamma| \setminus \{ x_0 \}$. It follows that $U = I(f)$, as required.
\end{proof}

In the remainder of the paper we write $\No$ for the set $\N\cup\{0\}$.

\begin{proof}[Proof of Theorem~\ref{theo:hasnounboundedcomponent}]
Suppose that $S=I(f)$, where $f : \R^d \to \R^d$ is a continuous function, and, as stated in the hypotheses of the theorem, suppose that $S$ is open with no unbounded components. Let $U$ be a component of $S$. Let $U'$ be the component of  $S$ containing $f(U)$. Clearly $U \ne U'$. Notice also that, by continuity, $f(\partial U)\subset \partial U'$.

Using these observations we can construct a sequence $(U_i)_{i\in\No}$ of components of $S$ such that $f(U_{i-1})\subset U_i$ and $f(\partial U_{i-1})\subset \partial U_i$, for $i\isnatural$. Clearly the sets $U_i$ are pairwise disjoint. Now choose any point $x_0$  in $\partial U_0$; because $x_0\notin I(f)$, there is a sequence $(k_i)_{i\isnatural}$ of positive integers and a positive constant $r$ such that $|f^{k_i}(x_0)|\leq r$, for $i\isnatural$. Define $V_i=U_{k_i}$, for $i\isnatural$. Since $f^{k_i}(x_0)\in \partial V_i$, we see that $\inf\{|x|\,:\,x \in V_i\} \leq r$, for $i\isnatural$. Also, $\sup\{|x|\,:\, x \in V_i\} \rightarrow +\infty$ as $i\rightarrow\infty$ because $f^{k_i}(U_0)\subset V_i$ and $ U_0$ is contained in the escaping set of $f$. This completes the proof of the theorem.
\end{proof}
%
%
%
%
\section{One dimension}\label{sec:R1}

In this section we prove various results about $I_1$-sets; later we will see that there are no comparable results for $I_d$-sets, where $d>1$. Before we begin we have some remarks about notation. One way in which the real line $\R$ is distinguished among the spaces $\R^d$, for $d\in\N$, is that it has two topological ends (two `infinities'), whereas all the other spaces have one end. A consequence of this is that  statements about convergence to infinity in one dimension can be ambiguous. So far we have followed the convention that if $(x_n)_{n\in\mathbb{N}}$ is a sequence in $\R^d$, then the statement $x_n\to\infty$ as $n\to\infty$ means that the sequence eventually leaves any compact subset of $\R^d$. We will continue to follow this convention, even when $d=1$. In the one-dimensional case, we write $x_n\to +\infty$ as $n\to\infty$ to mean that the sequence $(x_n)_{n\in\mathbb{N}}$ eventually leaves any interval of the form $(-\infty,K)$, and we write $x_n\to -\infty$ as $n\to\infty$ if $-x_n\to+\infty$ as $n\to\infty$. (In fact, we have used this notation for convergence to $+\infty$ a couple of times already.)

\begin{proof}[Proof of Theorem~\ref{theo:R1irrational}]

We define a continuous function $f : \R \to \R$ as follows (see Figure~\ref{fig1}):
\[
f(x) = \begin{cases}
   x + 1,    &\qfor x < \frac{3}{4},\\
   4n - 3x, &\qfor n - \frac{1}{4} \leq x < n,  n\isnatural,\\
   5x - 4n, &\qfor n \leq x < n + \frac{1}{4}, n \isnatural, \\
   4n + 2 - 3x, &\qfor n + \frac{1}{4} \leq x < n + \frac{1}{2}, n\isnatural,\\
   5x - 4n - 2, &\qfor n + \frac{1}{2} \leq x < n + \frac{3}{4}, n \isnatural.
\end{cases}
\]

\begin{figure}[ht]
\includegraphics[width=.6\textwidth]{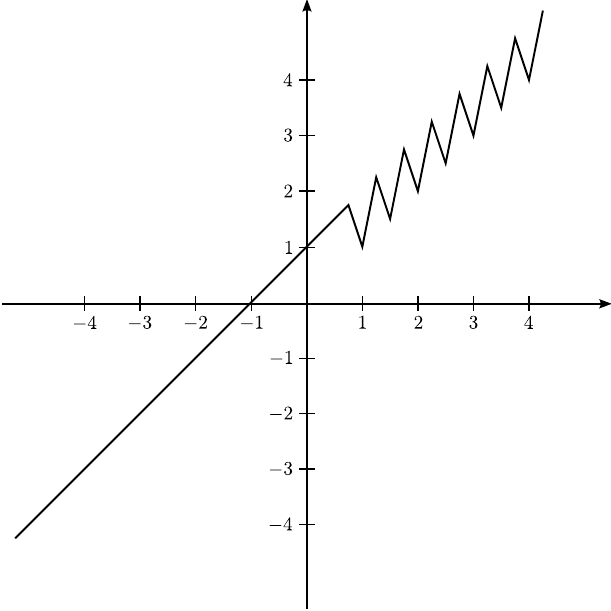}
\caption{Graph of the function $f$}
\label{fig1}
\end{figure}

Notice that $f(x)\geq x$, for $x\in\R$, and if $x>\tfrac34$, then for all but countably many values of $x$, the derivative $f'(x)$ exists and satisfies  $|f'(x)| \geq 3$.

Suppose that $x = n + a$, where $n\isnatural$ and $a \in \left\{-\frac{1}{4}, \frac{1}{4}\right\}$. Then $f(x) = x+1$, and it follows that $x \in I(f)$. On the other hand, suppose that $x = n + b$, where $n\isnatural$ and $b \in \left\{0, \frac{1}{2}\right\}$. Then $f(x) = x$, so $x \notin I(f)$.

Next, suppose that $\Delta$ is a bounded open interval in $\R$. It is easy to see that there exists a positive integer $N$ such that the interval $f^N(\Delta)$ lies in $\left(\frac{3}{4}, +\infty\right)$ and has diameter greater than $\frac{1}{2}$. It follows that $f^N(\Delta)$ meets both $I(f)$ and $\R\setminus I(f)$. Hence $\Delta$ meets both $I(f)$ and $\R\setminus I(f)$, so $I(f)$ and $\R\setminus I(f)$ are dense in $\R$.

Let $y$ be a point of $\R\setminus I(f)$. As $y\leq f(y)\leq f^2(y)\leq \dotsb$, the iterates $f^k(y)$ must be bounded, so there exists a fixed point $p$ of $f$ such that $f^k(y)\rightarrow p$ as $k\rightarrow\infty$. If $f^k(y)$ lies within a distance of $\tfrac14$ from $p$, but is not equal to $p$, then $f^{k+1}(y)$ is greater than $p$, which is impossible. Therefore $f^m(y)= p$, for some positive integer $m$. We deduce that $\R\setminus I(f)$ is countable.

To summarise, the set $\R\setminus I(f)$ is a countable, dense subset of the real line. It  follows that there exists a homeomorphism $\phi : \overline{\R} \to \overline{\R}$ with $\phi(\infty)=\infty$ and $\phi(\R\setminus I(f)) =  \mathbb{Q}$. We deduce that
\[
I(\phi \circ f \circ \phi^{-1})=\R\setminus \Q,
\]
and this completes the proof.
\end{proof}

We use the following two lemmas in the proof of Theorem~\ref{theo:R1uncountable}. Recall that $\No=\N\cup\{0\}$.

\begin{lemma}\label{lemm:sequence}
Suppose that $f : \R^d \to \R^d$ is a continuous function, $(n_k)_{k\in\No}$ is a sequence of positive integers, and $(E_k)_{k\in\No}$ is a sequence of non-empty compact subsets of $\R^d$ such that $E_{k+1} \subset f^{n_k}(E_k)$, for $k\in\No$. Then there exists $y\in E_0$ such that $f^{s_k}(y) \in E_{k+1}$, where $s_k=\sum_{i=0}^{k}n_i$, for $k\in\No$.
\end{lemma}
\begin{proof}
Let
\[
F_k = \{x\in E_0\,:\, f^{s_i}(x)\in E_{i+1}, i=0,1,\dots,k\},\quad\text{for}\quad k\in\No.
\]
Then $(F_k)_{k\in\No}$ is a decreasing sequence of non-empty compact sets, so $F = \bigcap_{k=0}^\infty F_k$ is non-empty. We can choose any point $y$ in $F$ to complete the proof.
\end{proof}

In the next lemma, we use the special case of Lemma~\ref{lemm:sequence} when each $n_k=1$ to construct points in an escaping set. Later on, we apply the more general version of Lemma~\ref{lemm:sequence} to construct a point $x$ for which the iterates $(f^k(x))_{k\in\No}$ have a bounded subsequence.

\begin{lemma}\label{jdf}
Let $f:\R\to\R$ be a continuous function. Suppose that there is an element $x$ of $I(f)$ such that the sequence $(f^n(x))_{n\in\N}$ is unbounded above. Then given any positive constant $R$ there exists a compact interval $J\subset (R,+\infty)$ such that $J\cap I(f)$ is uncountable.
\end{lemma}
\begin{proof}
Observe that no two of the iterates $(f^k(x))_{k\in\No}$ are equal. For each $k\in\No$, let $B_k^+=[f^{k}(x), f^{p_k}(x)]$, where $p_k$ is such that
\[
f^{p_k}(x) = \min_{m\in\No} \{f^m(x)\, :\, f^m(x) > f^{k}(x) \}.
\]
Similarly, let $B_k^-=[f^{q_k}(x), f^{k}(x)]$, where $q_k$ is such that
\[
f^{q_k}(x) = \max_{m\in\No} \{f^m(x)\,:\, f^m(x) < f^{k}(x) \}.
\]
(If the set $\{m\in\No\,:\, f^m(x) < f^{k}(x) \}$ is empty, which can happen for at most one value of $k$, then we define $B_k^-=(-\infty,f^k(x)]$.) Let $\mathfrak{B}$ denote the collection of all the intervals $B_k^+$ and $B_k^-$. These intervals cover $\R$. All of them are bounded on the right, and at most one is unbounded on the left. Any two intervals from the collection $\mathfrak{B}$ are either disjoint or else intersect at one of their end points.

Let $C = (C_k)_{k\in\No}$ be a sequence of intervals from $\mathfrak{B}$. We say that $C$ is a \emph{covering sequence starting at $C_0$} if $C_{k+1} \subset f(C_k)$, for $k\in\No$, and $\inf\{|x|\,:\,x\in C_k\}\rightarrow +\infty$ as $k\rightarrow\infty$. Observe that, for each $k \in \No$, $f(B_k^+)$ contains either $B_{k+1}^+$ or $B_{k+1}^-$, and $f(B_k^-)$ contains either $B_{k+1}^+$ or $B_{k+1}^-$. Thus, if $C_0 \in \mathfrak{B}$, then there exists at least one covering sequence starting at $C_0$.

Suppose that there exists more than one covering sequence starting at each element of $\mathfrak{B}$. Choose any interval $C_0$ of $\mathfrak{B}$ that satisfies $C_0\subset (R,+\infty)$. It is easy to see that there are uncountably many different covering sequences starting at $C_0$. We can apply Lemma~\ref{lemm:sequence} (with $n_k=1$, for $k\in\N$) to obtain an uncountable collection of points in $C_0\cap I(f)$. The lemma follows on choosing $J=C_0$.

Let us assume instead that there is a set $C_0$ in $\mathfrak{B}$ for which there is a unique covering sequence, $C = (C_k)_{k\in\No}$ say, starting at $C_0$. By starting from $C_1$ instead of $C_0$, if necessary, we can assume that $C_0$ is bounded.  If $f(C_0) \ne C_1$, then $f(C_0)$ contains two elements of $\mathfrak{B}$, so there is more than one covering sequence starting at $C_0$, contrary to assumption. Hence $f(C_0) = C_1$, and for the same reason $f(C_k) = C_{k+1}$, for $k\isnatural$.

Let $f^{a}(x)$ and $f^b(x)$ be the end points of $C_0$, where $0\leq a<b$. Then $f^{a+k}(x)$ and $f^{b+k}(x)$ are the end points of $C_k$, for $k\in\No$. Suppose, in order to reach a contradiction, that $f^{b+k}(x)<f^{a+k}(x)$, for $k\in\No$. Let $m=b-a$. By choosing $k=qm+r$, where $r\in\{0,1,\dots,m-1\}$, for values $q=0,1,2,\dotsc$ in turn, we see that
\[
f^{a+r}(x)>f^{a+m+r}(x)>f^{a+2m+r}(x)>\dotsb. 
\]
Therefore the sequence $(f^n(x))_{n\in\No}$ is bounded above, which is a contradiction. Thus, on the contrary, there is an integer $k\in\No$ for which $f^{a+k}(x)<f^{b+k}(x)$, so that $C_k=[f^{a+k}(x),f^{b+k}(x)]$. By replacing $x$ with $f^{a+k}(x)$, and relabelling our covering sequence, we can assume that $C_0=[x,f^m(x)]$. (Making this change increases the size of a finite number of the intervals from the collection $\mathfrak{B}$.) Let $g=f^m$, and let $D_k=C_{mk}$, for $k\in\No$. It follows that $D_k = [g^k(x),g^{k+1}(x)]$ and $g(D_k) = D_{k+1}$, for $k\in\No$. Notice that $I(f)=I(g)$.

We can assume that the set of positive fixed points of $g$ is unbounded above, as otherwise $(K,+\infty)\subset I(f)$, for some number $K$. Let $(k_n)_{n\in\No}$ be an increasing sequence of positive integers such that $D_{k_n}$ contains a fixed point $t_n$ of $g$, for $n\in\No$.

Let $\underline{s} = s_0 s_1 s_2 \ldots$ be an infinite sequence of $0$s and $1$s. There are uncountably many such sequences. We construct a sequence $(E_n)_{n\in\No}$, inductively, for use in Lemma~\ref{lemm:sequence}. Set $E_k = D_k$, for $0~\leq~k~\leq~k_0$. Note that
\[
[t_0, g^{k_0+1}(x)] \cup D_{k_0+1} \subset g([t_0, g^{k_0+1}(x)]) \subset g(D_{k_0}).
\]
If $s_0 = 0$, then define
\[
E_{k_0+1}=[t_0, g^{k_0+1}(x)]\quad\text{and}\quad E_{k_0+2} = D_{k_0+1},
\]
and if $s_0=1$, then define
\[
E_{k_0+1} = D_{k_0+1}.
\]

Suppose now that the values $s_0,\dots,s_p$ have been used to define intervals $E_0,\dots,E_n$, where $E_n=D_q$ for some positive integer $q$. Let $j\isnatural$ be such that $q < k_{j}$. Set $E_{n+k} = D_{q+k}$, for $0 < k \leq k_{j} - q$. Note that
\[
[t_j, g^{k_j+1}(x)] \cup D_{k_j+1} \subset g([t_j, g^{k_j+1}(x)]) \subset g(D_{k_j}).
\]
If $s_{p+1} = 0$, then define
\[
E_{n+k_{j} - q+1} = [t_j, g^{k_j+1}(x)]\quad\text{and}\quad E_{n+k_{j} - q+2} = D_{k_j+1},
\]
and if $s_{p+1}=1$, then define
\[
E_{n+k_{j} - q+1} = D_{k_j+1}.
\]
This completes the definition of the sequence of compact intervals $(E_n)_{n\in\No}$.

Applying Lemma~\ref{lemm:sequence}, we obtain a point $y$ in $E_0$ such that $g^n(y) \in E_n$, for $n\isnatural$. Because $g^n(y)\to+\infty$ as $n\to\infty$, we see that $y \in D_0 \cap I(f)$. Each sequence $\underline{s} = s_0 s_1 \ldots$ generates a different sequence of intervals $(E_n)_{n\in\No}$, and thereby gives rise to a different point in $D_0\cap I(f)$. Therefore $D_0 \cap I(f)$ is uncountable. It follows that $D_k \cap I(f)$ is uncountable, for $k\isnatural$, so the lemma is established on choosing $J=D_k$, for a sufficiently large positive integer $k$.
\end{proof}

\begin{proof}[Proof of Theorem~\ref{theo:R1uncountable}]
Suppose that $f : \R \to \R$ is a continuous function, and that $x \in I(f)$. Let $U$ be a neighbourhood of $x$. We can assume that $U$ meets $\R\setminus I(f)$, as otherwise there is nothing to prove. Choose a point $v$ in $U \cap (\R\setminus I(f))$. Then there exists $R>0$ such that $|f^{n_k}(v)|<R$, for some increasing sequence of positive integers $(n_k)_{k\in\N}$. We know that $|f^{n_k}(x)|\to+\infty$ as $k\to\infty$, and, after conjugating $f$ by the function $x\mapsto -x$ if necessary, we can assume that $f^{m_k}(x)\to +\infty$ as $k\to\infty$, for some subsequence $(m_k)_{k\in\N}$ of $(n_k)_{k\in\N}$. By Lemma~\ref{jdf}, there is a compact interval $J\subset (R,+\infty)$ such that $J\cap I(f)$ is uncountable. Let $L$ denote the closed interval with end points $v$ and $x$. Then we can choose a sufficiently large positive integer $k$ such that $J\subset f^{m_k}(L)$. For each point $y$ in $J \cap I(f)$, there exists a point $z$ in $L$ such that $f^{m_k}(z) = y$, so $z \in I(f)$. Hence $U \cap I(f)$ is uncountable, as required.
\end{proof}

\begin{proof}[Proof of Theorem~\ref{theo:R1open}]
Let $f : \R \to \R$ be a continuous function. Suppose, in order to reach a contradiction, that one of the components of the complement of $I(f)$ is \emph{not} closed. Then, after conjugating $f$ by $x\mapsto -x$ if necessary, we can assume that there is an open interval $(a,b)$ in $\R\setminus I(f)$ such that $b\in I(f)$. Since $a \notin I(f)$, there is a sequence $(n_k)_{k\isnatural}$ of positive integers and $R > 0$ such that $|f^{n_k}(a)| \leq R$, for $k\isnatural$. On the other hand, we know that $|f^{n_k}(b)|\to +\infty$ as $k\to\infty$. By passing to a subsequence, we can in fact suppose that either $f^{n_k}(b)\to+\infty$ as $k\to\infty$ or $f^{n_k}(b)\to -\infty$ as $k\to\infty$; let us assume the former (the other case can be handled in a similar way). Now, for values of $k$ sufficiently large that $f^{n_k}(b)>R$, we have
\[
[R,f^{n_k}(b)) \subset f^{n_k}([a,b)) \subset \R\setminus  I(f).
\]
This is a contradiction, because for sufficiently large values of $k$, the intervals  $[R,f^{n_k}(b))$ contain iterates of $b$ (so they intersect $I(f)$). Therefore, contrary to our assumption, all components of the complement of $I(f)$ are closed.
\end{proof}

%
%
%
%
%
%
\section{examples}\label{sec:examples}

In this section we provide examples that demonstrate differences between the class of $I_1$-sets and the class of $I_2$-sets. Many of our examples are of sets that appeared in Figure~\ref{fig:sets} from the introduction. We shall see that sets (a), (c) and (e) from that figure \emph{are} $I_2$-sets, and sets (b), (d) and (f) are \emph{not} $I_2$-sets.

By Theorem~\ref{theo:R1open}, it is not possible for an $I_1$-set to have an open complementary component. The following example shows that this is not true of $I_2$-sets.

\begin{example}\label{example.halfplane}
Let $f : \R^2 \to \R^2$ be the continuous function given by
\[
f(x,y) =
\begin{cases}
(x, y+1),      &\qfor x \geq 0, \\
(x, y),          &\qfor x \leq -1,\\
(x(x + 2), y + x + 1), &\qfor -1 < x < 0.
\end{cases}
\]
It can be shown by an elementary calculation that $I(f)$ is the closed half-plane $\{(x, y) : x \geq 0 \}$.
\end{example}

In contrast to this example, we observe that the escaping set of a transcendental entire function cannot be closed. This follows from the observation that if, on the contrary, $f$ is a transcendental entire function and $I(f)$ is closed, then the Julia set of $f$, which is equal to $\partial I(f)$, is contained in $I(f)$, which contradicts the well-known fact that the Julia set contains periodic points of $f$.

It follows from Theorem~\ref{theo:R1uncountable} that $\Z$ is not an $I_1$-set. In the next two examples we will show that both  $\Z \times \{0\}$ and its complement are $I_2$-sets.

\begin{example}\label{example.Z}
Let $f : \R^2 \to \R^2$ be the continuous function given by
\[
f(x, y) =
\begin{cases}
  (x+1, |2y| + \sin^2 (\pi x)), &\qfor                       y \leq 1, \\
  (x + 2 - y, (2-y)\sin^2 (\pi x) + 2), &\qfor              1 < y \leq 2, \\
  (x, y), &\qfor y > 2.
\end{cases}
\]
Since $f(n,0) = (n+1,0)$, for $n\in\Z$, we have $\Z \times \{0\} \subset I(f)$. Now suppose that $(x, y) \notin \Z \times \{0\}$. Then $f(x,y)$ lies in the upper half-plane $\{(x, y) : y \geq 0 \}$. It can then be seen, by a calculation, that $f^{k+1}(x,y) = f^k(x,y)$, for all sufficiently large values of $k$, so $(x,y) \notin I(f)$. Therefore $I(f)=\Z \times \{0\}$.
\end{example}

\begin{example}\label{example.Zcomp}
Let $g : \R^2 \to \R^2$ be the continuous function given by
\[
g(x, y) = f(x, y) - (1, 0),
\]
 where $f$ is the function in Example~\ref{example.Z}. It can be shown, using similar methods to those of  Example~\ref{example.Z}, that $I(g) = \R^2 \setminus (\Z \times \{0\})$.
\end{example}

The $I_2$-set $\R^2 \setminus (\Z \times \{0\})$ from the previous example is shown in Figure~\ref{fig:sets}(a). The set $\bigcup_{n\in\mathbb{Z}} D(n,\tfrac14)$, which is shown in Figure~\ref{fig:sets}(b), is not an $I_2$-set, because it does not satisfy the conditions of Theorem~\ref{theo:hasnounboundedcomponent}.

The next two examples are each about classes of sets that \emph{do} satisfy the conditions of Theorem~\ref{theo:hasnounboundedcomponent}. One of the sets from Example~\ref{example.open} is shown in Figure~\ref{fig:sets}(c): we will see that this \emph{is} an $I_2$-set. (Actually, Figure~\ref{fig:sets}(c) is a reflection in the $x$-axis of one of the sets from Example~\ref{example.open}, but this qualification is insignificant.) Example~\ref{example.open} shows that, in contrast to the one-dimensional case, there are open $I_2$-sets without unbounded components. One of the sets from Example~\ref{example.nastyone} is shown in Figure~\ref{fig:sets}(d): we will see that this \emph{is not} an $I_2$-set.

\begin{example} \label{example.open}
In this example (and only in this example), we use polar coordinates: $(r,\theta)$ represents the point in the plane with modulus $r$ and argument $\theta$. We represent the origin by $(0,\theta)$, for any value of $\theta$.

Let $(\rho_n)_{n\isnatural}$, $(\phi_n)_{n\isnatural}$, and $(\phi'_n)_{n\isnatural}$ be increasing sequences of positive real numbers such that $\rho_n\rightarrow +\infty$ as $n\rightarrow\infty$, and
\[
0 < \phi_1 < \phi'_1 < \phi_2 < \phi'_2 < \dots < 2\pi.
\]
Let
\[
A_n = \{ (r, \theta) : 0< r < \rho_n, \ \phi_n < \theta < \phi'_n \}, \quad\qfor n\isnatural;
\]
these are disjoint, bounded open sectors. Let $S = \bigcup_{n\isnatural} A_n$. We will construct a continuous function $f:\R^2\to\R^2$ such that $I(f)=S$, so that $S$ is an $I_2$-set.

First define $f(x) = 0$, for $x \notin S$, so $I(f) \subset S$. Next choose any number $\varepsilon$ in $(0, 1/2)$. The idea is to define $f$ in $S$ such that the following three properties hold: first,
\[
f(A_n) \subset A_{n+1}, \quad\qfor n\isnatural;
\]
second,
\[
\quad f\left(\frac{\rho_n}{1+\varepsilon}, \frac{\phi_n + \phi'_n}{2}\right) = \left(\frac{\rho_{n+1}}{1+\varepsilon}, \frac{\phi_{n+1} + \phi'_{n+1}}{2}\right), \quad\qfor n\isnatural;
\]
and third, the orbit of each point in $S$ eventually joins the orbit
\[
\left(\frac{\rho_1}{1+\varepsilon}, \frac{\phi_1 + \phi'_1}{2} \right)\to \left(\frac{\rho_2}{1+\varepsilon}, \frac{\phi_2 + \phi'_2}{2}\right) \to \cdots \to \left(\frac{\rho_n}{1+\varepsilon}, \frac{\phi_n + \phi'_n}{2} \right) \to \cdots.
\]
It then follows that $S = I(f)$.

To complete the definition of the function $f$, we use the function $h : [0,1] \to [0, 1/(1+\varepsilon)]$ given by
\[
h(t) =
\begin{cases}
t/\varepsilon, &\qfor 0 \leq t \leq \varepsilon/(1+\varepsilon), \\
1/(1+\varepsilon), &\qfor \varepsilon/(1+\varepsilon) < t < 1/(1+\varepsilon), \\
(1-t)/\varepsilon, &\qfor 1/(1+\varepsilon) \leq t \leq 1.
 \end{cases}
\]
Clearly $h$ is continuous, and $h(0) = h(1) = 0$. Moreover, if $x \in (0,1)$, then we have that $h^k(x) = 1/(1+\varepsilon)$, for all sufficiently large values of $k$.

For each $n\isnatural$, and for $(r, \theta) \in A_n$, we define
\[
f(r, \theta) = \left((1+\varepsilon)\rho_{n+1}h\left(\frac{r}{\rho_n}\right)h\left(\frac{\theta - \phi_n}{\phi'_n - \phi_n}\right), \frac{\phi_{n+1} + \phi'_{n+1}}{2} \right).
\]

It is a straightforward calculation to see that $f : \R^2 \to \R^2$ is continuous, and it satisfies the first two of the three properties described above. To see that the third property also holds, choose any point $x$ in $S$. Without loss of generality we can assume that $x \in A_1$. Write $x = (r_1, \theta_1)$ and, in general, write $f^n(x) = (r_{n+1}, \theta_{n+1})$, for $n\isnatural$, where $r_n\geq 0$ and $0\leqslant \theta_n<2\pi$. Observe that $\theta_n = (\phi_n + \phi'_n)/2$, for $n > 1$. Let $\alpha_n = r_n / \rho_n$, for $n\isnatural$. Then $\alpha_{n+1} = h(\alpha_n)$, for $n>1$. Therefore, for sufficiently large values of $n$,  we have $\alpha_{n} = 1/(1+\varepsilon)$, and hence $r_n = \rho_n/(1+\varepsilon)$, as required. This completes our construction of the function $f$.
\end{example}

Our next example shows that an open set can satisfy the conditions of Theorem~\ref{theo:hasnounboundedcomponent} without being the escaping set of a continuous function.

\begin{example}\label{example.nastyone}
Let $D$ be a bounded plane domain, and let $(U_n)_{n\isnatural}$ be a sequence of bounded plane domains such that
\begin{enumerate}[label=(\roman*)]
\item $\overline{U_n} \cap \overline{U_m} = \emptyset$, for $n \ne m$;
\item $U_n \cap D = \emptyset$, for $n\in\N$;
\item for each $n\in\N$ there exists a point $p_n$ such that $\partial U_n\cap \partial D = \{p_n\}$;
\item  $\sup\{|x|\,:\, x \in U_n\}\to+\infty$ as $n\rightarrow\infty$.
\end{enumerate}
Let
\[
S = D \cup \bigcup_{n\isnatural} U_n.
\]
Clearly $S$ satisfies the conditions of Theorem~\ref{theo:hasnounboundedcomponent}. A picture of such a set is given in Figure~\ref{fig:sets}(d).

We claim that $S$ is not an $I_2$-set. To see why this is so, suppose that $S = I(f)$ where $f : \R^2 \to \R^2$ is a continuous function. As we saw earlier, if $W$ is a component of $I(f)$, then there is another, different component $W'$ of $I(f)$ such that $f(W) \subset W'$ and $f(\partial W) \subset \partial W'$. Therefore, after relabelling the sets $U_i$, we can assume that $f(\partial D) \subset \partial U_0$, in which case $f(p_0) \in \partial U_0$. We can also assume that  $f(\partial U_0) \subset \partial U_1$, and hence that $f(p_0) \in \partial U_1$. This is a contradiction, because $\partial U_1$ is disjoint from $\partial U_0$. Therefore, contrary to our assumption, $S$ is not an $I_2$-set.
\end{example}

Our last two examples are illustrated by Figures~\ref{fig:sets}~(e) and~(f), in that order. Both are examples of unbounded open sets from which infinity is not accessible; the first is an $I_2$-set and the second is not an $I_2$-set. In the first example we use the notation $\pi_1(\xi)$ and $\pi_2(\xi)$ to denote the $x$- and $y$-coordinates, respectively, of a point $\xi\in\R^2$.

\begin{example}\label{example:snake}
Let  $A_n = \sum_{k=0}^n 1/(k+1)$, for $n\in\No$. We define an unbounded curve $\Gamma : [0,+\infty) \to \R^2$, as follows. Given any number $t\in[0,+\infty)$, we define $n\in\No$ and $\lambda\in[0,1)$ by the equation $t=n+\lambda$. Then we set
\[
\Gamma(t) =
\begin{cases}
   \left(\lambda A_{n/2}, (2-\lambda)/2^{n+1}\right), &\text{for $n$ even}, \\
   \left((1-\lambda)A_{(n-1)/2}, (2-\lambda)/2^{n+1}\right),& \text{for $n$ odd}.
\end{cases}
\]
One can check that $\Gamma$ satisfies the inequality
\begin{equation}\label{kdl}
|\Gamma(s)-\Gamma(t)|<4A_n|s-t|,\quad\text{when}\quad s,t\in[n,n+2],\quad n\geq 0.
\end{equation}

For simplicity we denote the image set $\Gamma([0,+\infty))$ by $\Gamma$, and we follow similar conventions for other unbounded curves in this example. The image set $\Gamma$ consists of a sequence of line segments, which join the points
\[
\left(0,1\right), \left(A_0,1/2\right), \left(0, 1/4\right), \left(A_1, 1/8\right), \ldots,  \left(0,1/2^{2n}\right), \left(A_n, 1/2^{2n+1}\right), \left(0, 1/2^{2n+2}\right), \ldots,
\]
in that order. The ``right-hand'' points, which are of the form $\left(A_n, 1/2^{2n+1}\right)$, correspond to odd integer values of $t$. The ``left-hand'' points, which are of the form $\left(0, 1/2^{2n}\right)$, correspond to even integer values of $t$. If $y \in (0, 1]$, then there exists a unique $x \geq 0$ such that $(x, y) \in \Gamma$. We denote this unique value $x$ by $\phi(y)$.

Let $U$ be the unbounded domain in $\R^2$ given  by
\[
U = \{ (x, y) : y \in (0, 1) \text{ and } |x - \phi(y)| < y \}.
\]
The set $U$ is enclosed by three curves: the horizontal line segment $\{(x,1) : -1\leq x\leq 1 \}$, an unbounded curve $\Gamma_l$ to the left of $\Gamma$, and an unbounded curve $\Gamma_r$ to the right of $\Gamma$. The two curves $\Gamma_l$ and $\Gamma_r$ both zigzag in the same fashion as $\Gamma$. We parameterise them in the same way as $\Gamma$, in the sense that, if $t\geq 0$, then
\[
\pi_2(\Gamma(t)) = \pi_2(\Gamma_r(t)) =\pi_2(\Gamma_l(t)),
\]
and
\[
\pi_1(\Gamma(t)) = \pi_1(\Gamma_r(t)) -  \pi_2(\Gamma(t)) = \pi_1(\Gamma_l(t)) + \pi_2(\Gamma(t)).
\]
The set $U$ is shown in Figure~\ref{haf}. It is unbounded, and infinity is not accessible from $U$.

\begin{figure}[ht]
\centering
\includegraphics{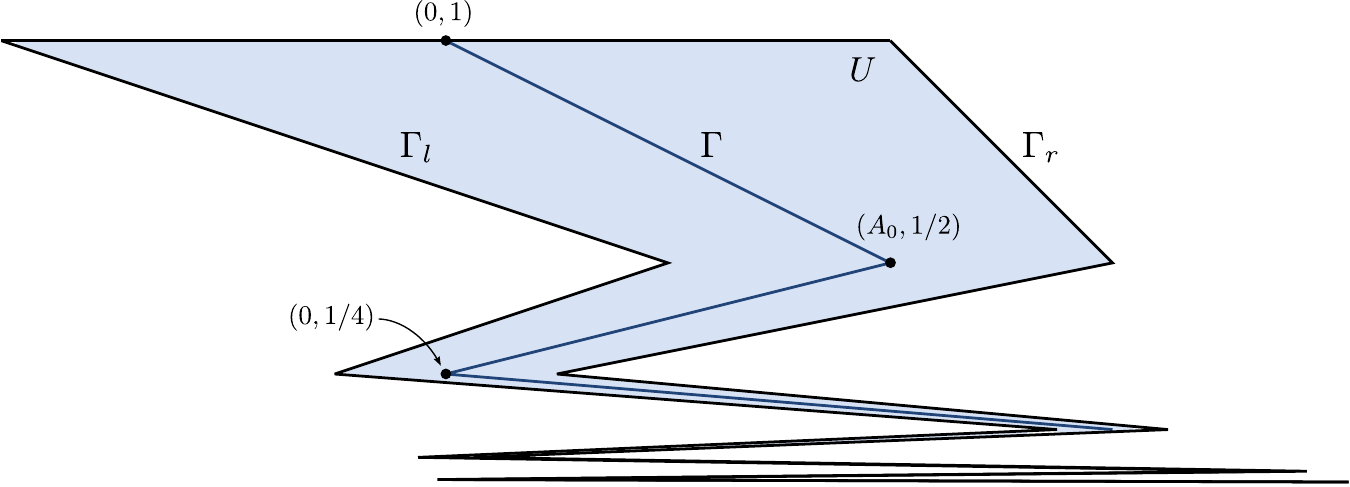}
\caption{An unbounded open set that is an $I_2$-set}
\label{haf}
\end{figure}

We must now define a continuous function $f : \R^2 \to \R^2$ with $I(f) = U$. To begin with, we define $f$ to be the identity function on the set
\[
V = \{ (x, y) : y\notin (0,1) \}.
\]

Next we define $f$ on $U$. To do this, we first define a function $\psi$ on $\Gamma \cup \Gamma_l \cup \Gamma_r$ that satisfies $\psi(\Gamma) = \Gamma$, $\psi(\Gamma_l) = \Gamma_l$ and $\psi(\Gamma_r) = \Gamma_r$. Roughly speaking, the function $\psi$ will be chosen so that it maps each ``zigzag'' of $\Gamma$ (and $\Gamma_l$ and $\Gamma_r$) downwards to the zigzag below. We will then define  $f$ on $U$ by interpolating between $\psi$ on $\Gamma$ and $\psi$ on $\Gamma_l \cup \Gamma_r$. 

 Let $\beta : [0,+\infty)\to[0,+\infty)$ be the continuous function given by
\[
\beta(t) =
\begin{cases}
n + 1, &\qfor t \in [n - 1/n, n],\\
(1+1/n)t, &\qfor t \in [n, n+1-1/(n+1)],
\end{cases}
\]
for $n\in\N$.  One can check that $\beta$ satisfies the inequality
\begin{equation}\label{iab}
|\beta(t)-(t+1)|\leq \frac{1}{n}, \quad\text{when}\quad t\in[n,n+1],\quad n\in\No.
\end{equation}
We also record two other important properties of $\beta$, namely that (a) $\beta(n) = n+1$, for $n\isnatural$, and (b) if $t>0$, then $\beta^k(t) \isnatural$, for all sufficiently large values of $k$.

Let
\[
h_1(t) =
\begin{cases}
4t, &\qfor t \in [0, 1/2], \\
2t+1, & \qfor t \in [1/2,+\infty),
\end{cases}
\]
and define
\[
\psi(\Gamma(t)) = \Gamma(h_1\circ \beta\circ h_1^{-1}(t)), \quad\text{for}\quad t\in[0,+\infty).
\]

The function $\psi$ is continuous on $\Gamma$, and it maps $\Gamma$ into itself. Using property (a) from above, we can see that $\psi$ maps each \emph{right-hand} vertex $\left(A_n, 1/2^{2n+1}\right)$ of $\Gamma$ to the next right-hand vertex $\left(A_{n+1}, 1/2^{2n+3}\right)$, for $n\in\N$. Next, by property (b), if $\xi \in \Gamma$, then $\psi^k(\xi)$ is one of the right-hand vertices of $\Gamma$, for all sufficiently large  $k$. We deduce that $\psi^k(\xi) \rightarrow \infty$ as $k\rightarrow\infty$, for $\xi \in \Gamma$. Finally, by combining inequalities \eqref{kdl} and \eqref{iab}, it can be shown that $|\psi(\Gamma(t))-\Gamma(t+2)|\to 0$ as $t\to\infty$, and hence that $|\psi(\Gamma(t))-\Gamma(t)|\to 0$ as $t\to\infty$; we omit the detailed calculations. 

The definition of $\psi$ on $\Gamma_l \cup \Gamma_r$ is similar to that on $\Gamma$. We define $\psi$ on $\Gamma_l$ and describe its properties; the definition and properties on $\Gamma_r$ are much the same but with $\Gamma_r$ replacing $\Gamma_l$. Let $h_2(t)=2t$, and define
\[
\psi(\Gamma_l(t)) = \Gamma_l(h_2\circ\beta\circ h_2^{-1}(t)\big), \quad\text{for}\quad t\in[0,+\infty).
\]

Then $\psi$ is continuous on $\Gamma_l$, and it maps $\Gamma_l$ into itself. It maps each \emph{left-hand} vertex of $\Gamma_l$ to the left-hand vertex below, and if $\xi \in \Gamma_l$, then $\psi^k(\xi)$ is one of the left-hand vertices of  $\Gamma_l$, for all sufficiently large $k$. Hence $\psi^k(\xi) \rightarrow (0, 0)$ as $k\rightarrow\infty$, for $\xi \in \Gamma_l$. Finally, we observe that $|\psi(\Gamma_l(t))-\Gamma_l(t)|\to 0$ as $t\to\infty$.

Next we define $f$ on $U$ in such a way that $f(x,y)=\psi(x,y)$ when $(x,y)\in \Gamma\cup\Gamma_l\cup\Gamma_r$ and $y<1/2$. To do this, it  is slightly easier to use an alternative coordinate system. Suppose that $(x,y)$ is a point with $0<y\leq 1$. Then, by the definitions of $\Gamma$ and $\phi$, there exists a unique real number $\eta$ such that $x = \phi(y) + \eta y$. We write $(x, y) = \floor*{\eta, y}$, and say that $\eta$ and $y$ are the \emph{$U$-coordinates} of the point $(x,y)$. We also say that $\eta$ is the \emph{$\eta$-coordinate} of the point $(x,y)$. Note that $(x,y) \in \Gamma$ if and only if $(x, y) = (\phi(y), y) = \floor*{0,y}$. Note also that $(x,y) \in \Gamma_l$ if and only if $(x, y) = \floor*{-1,y}$, and $(x,y) \in \Gamma_r$ if and only if $(x, y) = \floor*{1,y}$.

Let $\alpha : [0,1] \to [0, 1]$ be the continuous function given by
\[
\alpha(t) =
\begin{cases}
0, &\qfor t \in [0, 1/2], \\
2t-1, & \qfor t \in [1/2,1].
\end{cases}
\]

We recall the definition $\pi_2(x,y)=y$. Now, choose any point $(x,y)$ in $U \cup \Gamma_l \cup \Gamma_r$. We write $(x,y)$ in the form $\floor*{\varepsilon\eta, y}$, where $\eta\in[0,1]$ and $\varepsilon$ is either $-1$ or $1$. Let
\[
f(\floor*{\varepsilon\eta, y})= \floor*{\varepsilon\eta', y'},
\]
where
\[
\eta'=(1-\alpha(y))\alpha(\eta)+\alpha(y)\eta,
\]
and
\[
y'=(1-\alpha(y))\big((1-\eta)\pi_2(\psi(\floor*{0, y}))+\eta\pi_2(\psi(\floor*{\varepsilon, y}))\big)+\alpha(y)y.
\]
This is a well-defined function, because $\eta'\in[0,1]$ and $y'\in (0,1]$. Clearly it is continuous on $U\cup\Gamma_l\cup\Gamma_R$. When $\eta=0$, we have $\eta'=0$, so $f$ maps $\Gamma$ into itself, and similarly it maps $\Gamma_l$ and $\Gamma_r$ into themselves. What is more, if $\eta<1$, then $\eta'<1$, so $f$ maps $U$ into itself. 

Next we prove that $f$ is continuous on $\overline{U}\cup V$. When $y=1$, we have $y'=1$ and $\eta'=\eta$, so $f$ coincides with the identity function. Observe that if $y<1/2$, then the definition of $f$ simplifies to 
\[
f(\floor*{\varepsilon\eta, y})= \floor*{\varepsilon\alpha(\eta), (1-\eta)\pi_2(\psi(\floor*{0, y}))+\eta\pi_2(\psi(\floor*{\varepsilon, y}))},
\]
in which case $y'<1/2$ also. Notice in particular that $f(x,y)=\psi(x,y)$ when $(x,y)\in \Gamma\cup\Gamma_l\cup\Gamma_r$ and $y<1/2$.  Now suppose that $\big((x_n,y_n)\big)_{n\in\N}$ is a sequence of points in $U\cup\Gamma_l\cup\Gamma_r$ that converges to a point $(u,0)$. Let $\varepsilon_n\eta_n\in[-1,1]$ be the $\eta$-coordinate of $(x_n,y_n)$, where $\varepsilon_n$ is $-1$ or $1$ and $\eta_n\in[0,1]$, so that $x_n=\phi(y_n)+\varepsilon_n\eta_ny_n$. We know that $x_n\to u$ and $y_n\to 0$ as $n\to\infty$, and it follows that $\phi(y_n)\to u$ and $\phi(y_n)+\varepsilon_n y_n\to u$ as $n\to\infty$. Hence $\floor*{0, y_n}\to (u,0)$ and $\floor*{\varepsilon_n, y_n}\to (u,0)$ as $n\to\infty$. It then follows that $\psi(\floor*{0, y_n})\to (u,0)$ and $\psi(\floor*{\varepsilon_n, y_n})\to (u,0)$ as $n\to\infty$. Furthermore, if we define
\[
(x_n',y_n') = (1-\eta_n)\psi(\floor*{0, y_n})+\eta_n\psi(\floor*{\varepsilon_n, y_n}),
\]
then $(x_n',y_n')\to (u,0)$ as $n\to\infty$, also. Let $\eta_n'\in[-1,1]$ be the $\eta$-coordinate of $(x_n',y_n')$, so that $x_n'=\phi(y_n')+\eta_n'y_n'$. Then we see that $\phi(y_n')\to u$ as $n\to\infty$. Hence
\[
f(x_n,y_n) = (\phi(y_n')+\varepsilon_n\alpha(\eta_n)y_n',y_n') \to (u,0)\quad\text{as}\quad n\to\infty.
\]
This completes our argument to show that $f$ is continuous on $\overline{U}\cup V$.

Let us now prove that $U\subset I(f)$ and $(\Gamma_l\cup\Gamma_r)\cap I(f)=\emptyset$. Choose any point  $\floor*{\varepsilon\eta, y}$ in $U\cup\Gamma_l\cup\Gamma_r$, where, as usual, $\varepsilon$ is $-1$ or $1$ and $\eta\in[0,1]$, and for the moment we assume that $1/2<y<1$. If $1/2<y<3/4$, then one can check that $y'<1/2$. If $3/4\leq y<1$, then one can check that $y'<\alpha(y)$. It follows that the $y$-component of $f^k(\floor*{\varepsilon\eta, y})$ is less than  $1/2$, for a sufficiently large positive integer $k$. Since $f$ coincides with $\psi$ on $\Gamma\cup\Gamma_l\cup\Gamma_r$ when $y<1/2$, we see that $(\Gamma_l\cup\Gamma_r)\cap I(f)=\emptyset$ and $\Gamma\subset I(f)$. Now, if  $\floor*{\varepsilon\eta, y}\in U$ and $y<1/2$, then $\eta'=\alpha(\eta)$. It follows that $f^k(\floor*{\varepsilon\eta, y})\in \Gamma$, for a sufficiently large positive integer $k$.  Hence $U\subset I(f)$.

It remains to define $f$ in the set
\[
W = \{(x, y) \notin \overline{U} : 0<y<1\}.
\]
Points $(x,y)$ in this set can also be written in $U$-coordinates, in the form $(x,y)=\floor*{\varepsilon\eta,y}$, where $\eta>1$ and $\varepsilon$ is $-1$ or $1$. We define
\[
f(\floor*{\varepsilon\eta,y}) = f(\floor*{\varepsilon,y}) + \floor*{\eta - \varepsilon, 0}.
\]
 We leave it for the reader to check that $f$ is continuous on $\R^2$, and that $I(f) = U$.
\end{example}

We finish, as promised, with an example of an unbounded open set that is not an $I_2$-set.

\begin{example}\label{example:failsnake}
Let $E = \{ (x, y) : y < 0, (x-5)^2 + y^2 < 1 \}$ and define $S = U \cup E$, where $U$ is the domain from Example~\ref{example:snake}. A stylised image of $S$ is shown in Figure~\ref{hxf}.

\begin{figure}[ht]
\centering
\includegraphics{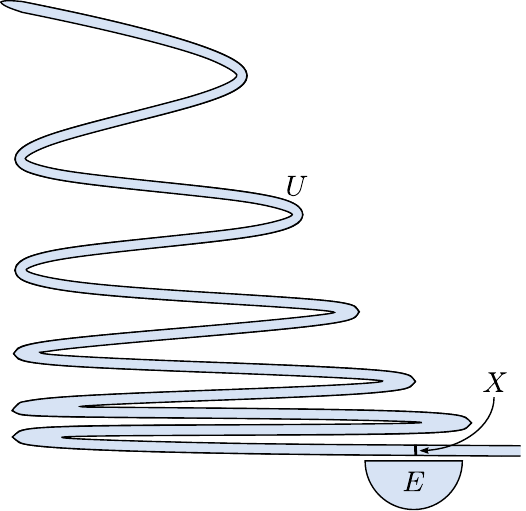}
\caption{An unbounded open set that is not an $I_2$-set}
\label{hxf}
\end{figure}

Let us suppose, in order to reach a contradiction, that $S=I(f)$ for some continuous function $f:\R^2\to\R^2$. As $f(S) \subset S$, we must have either $f(E)\subset E$ or $f(E)\subset U$. In fact, the former case cannot arise, because $E$ is bounded and lies in the escaping set. The closure of $E$ is compact, and $E$ is connected, so $f(E)$ is a bounded, connected subset of $U$. It follows that there is a positive constant $k$ such that
\[
f(E) \subset \{(x,y)\in U\,:\, y \geq k\}.
\]
Choose $\varepsilon>0$ such that if $\text{dist}((x,y),E)\leq \varepsilon$, then the $y$-component of $f(x,y)$ is greater than $k/2$. Let us assume also that $\varepsilon<k/2$.

Now choose a vertical cross section $X$ of $U$, in the obvious way. The set $X$ is simply a small vertical line segment, contained wholly within $U$, which splits $U$ into two components: a bounded component $L$ and an unbounded component $R$. We can choose $X$ so that each point of $X$ lies within a distance $\varepsilon$ of $E$.

Let $\gamma:[0,1]\to U$ be a continuous path such that $\gamma([0,1)) \subset L$, such that $\gamma(1)\in X$, and finally such that $f^n(\gamma(0))\in X\cup R$, for $n\isnatural$; it is easy to see that such a path exists. Let $|\gamma|$ denote the image of $\gamma$ in $U$. We call sets $|\gamma|$ that arise in this way \emph{restricted paths}.

Suppose that $\gamma$ is a restricted path, and $p=\gamma(0)$ and $q=\gamma(1)$. We know that the $y$-component of $f(q)$ is greater than or equal to $k/2$, so $f(q)\in L$. Let $N$ be the largest positive integer such that $f^N(q)\in L$. Notice that $f^n(p)\in X\cup R$, for $n\in\N$. Consider the path $\delta(t)=(f^N\circ\gamma)(1-t)$. It satisfies $\delta(0)\in L$ and $\delta(1)\in X\cup R$. Let $s\in (0,1]$ be the smallest positive value such that $\delta(s)\in X$. Finally, let
\[
\phi:[0,1]\to U,\quad \phi(t) = \delta(st).
\]
It follows by the definition of the integer $N$ that $|\phi|$ is a restricted path, and that 
\[
|\phi| \subset |\delta|=f^N(|\gamma|).
\]
We have shown that \emph{given any restricted path $|\gamma|$, there is another restricted path $|\phi|$ and a positive integer $N$ such that $|\phi|\subset f^N(|\gamma|)$.}

Using this observation, we can construct a sequence of restricted paths $(\alpha_i)_{i\isnatural}$ and a sequence of positive integers $(n_i)_{i\isnatural}$ such that
\[
|\alpha_{i+1}|\subset f^{n_i}(|\alpha_{i}|), \quad\qfor i\isnatural.
\]
As the sets $|\alpha_i|$ are compact, it follows from Lemma~\ref{lemm:sequence} that there is a point $x$ in $|\alpha_1|$ such that the sequence of iterates $(f^i(x))_{i\isnatural}$ has a bounded subsequence. This is a contradiction, because $x \in I(f)$. Thus, contrary to our earlier assumption, $S$ is not an $I_2$-set.
\end{example}

\itshape{Acknowledgement:} \normalfont We are grateful to John Osborne for a number of helpful comments.
%
%
%
%
%
%
\bibliographystyle{acm}

\end{document}